\documentclass[12pt,reqno]{amsart}

\usepackage{setspace}
\usepackage{makecell}

\usepackage[T1]{fontenc}
\usepackage{ae,aecompl}
\usepackage{graphics,cancel,graphicx}
\usepackage{epsfig,psfrag,manfnt}
\usepackage{mathrsfs,amsmath}
\usepackage{graphicx,mathabx}
\usepackage{bbm,subfigure,rotating,bm,float,mathdots,wasysym,scalerel,color,xcolor}
\usepackage{mathrsfs,setspace}
\usepackage{graphics,cancel,graphicx,mathabx}
\usepackage{epsfig,psfrag,manfnt}
\usepackage{bbm,subfigure,rotating,bm,float,mathdots,wasysym,scalerel}
\usepackage{array}
\usepackage{comment}

\definecolor{cerulean}{rgb}{0.0, 0.48, 0.65}
\definecolor{burntorange}{rgb}{0.8, 0.33, 0.0}
\definecolor{green}{rgb}{0.3,0.6,0.3}
\definecolor{red}{rgb}{1, 0.2, 0.1}

\usepackage[T1]{fontenc}
\usepackage{ae,aecompl}
\usepackage{mathrsfs,setspace}
\usepackage{graphics,cancel,graphicx,mathabx}
\usepackage{epsfig,psfrag,manfnt}
\usepackage{bbm,subfigure,rotating,bm,float,mathdots,wasysym,scalerel,yfonts}
\usepackage{array}

\usepackage{relsize}
\usepackage{exscale,xparse}

\usepackage{tikz-cd}
\usepackage{tikz}
\usetikzlibrary{positioning} 

\usepackage{array}
\usepackage[hidelinks]{hyperref}
\usepackage{xcolor}
\hypersetup{
    colorlinks,
    linkcolor={burntorange!50!black},
    citecolor={blue!50!black},
    urlcolor={blue!80!black}
}

\usepackage{tikz}
\usepackage[framemethod=TikZ]{mdframed}
\mdfsetup{skipabove=\topskip,skipbelow=\topskip}
\tikzstyle{titregris} =
    [draw=gray, thick, fill=gray,%
        text=black, rectangle,  
        right,minimum height=.7cm]
\pgfdeclarehorizontalshading{exercisebackground}{100bp}{color(0bp)=(gray!40);color(100bp)=(black!5)}
\makeatletter
\def\mdf@@exercisepoints{}
\define@key{mdf}{exercisepoints}{
    \def\mdf@exercisepoints{#1}
}
\def\mdf@@myboxedtitle{}
\define@key{mdf}{myboxedtitle}{
    \def\mdf@myboxedtitle{#1}
}
\mdfdefinestyle{exercisestyle}{%
    outerlinewidth=1em,%
    outerlinecolor=white,%
    leftmargin=-1em,%
    rightmargin=-1em,%
    middlelinewidth=1.2pt,%
    roundcorner=5pt,%
    linecolor=gray,%
    apptotikzsetting={\tikzset{mdfbackground/.append style={gray!30}}},
    innertopmargin=1.2\baselineskip,
    skipabove={\dimexpr0.5\baselineskip+\topskip\relax},
    skipbelow={-1em},
    needspace=3\baselineskip,
    singleextra={%
        \node[titregris,xshift=1cm] at (P-|O) %
            {~\mdf@frametitlefont{\mdf@myboxedtitle}~};
},%
    firstextra={%
\node[titregris,xshift=1cm] at (P-|O) %
{~\mdf@frametitlefont{\mdf@myboxedtitle}~};
},
}
\makeatother

\NewDocumentCommand{\qfrac}{smm}{%
  \dfrac{\IfBooleanT{#1}{\vphantom{\big|}}#2}{\mathstrut #3}%
}

\usepackage{adjustbox}

\usepackage{relsize}
\usepackage{exscale,xparse}

\usepackage{tikz-cd}
\usepackage{tikz}
\usetikzlibrary{positioning} 

\usepackage{array}

\usepackage{hyperref}

\usepackage[all,cmtip]{xy}
\usepackage{amssymb,amsmath,amsfonts,amsthm,color}

\let\oldsum\sum
\renewcommand{\sum}{\mathop{\scalebox{0.81}{$\oldsum$}}\limits}

\newcommand{\calM}{\mathcal{M}}

\newcommand{\calO}{\mathcal{O}}

\newcommand{\mC}{\mathbb{C}}
\newcommand{\mD}{\mathbb{D}}

\newcommand{\mN}{\mathbb{N}}

\newcommand{\mR}{\mathbb{R}}

\newcommand{\mT}{\mathbb{T}}

\newcommand{\mZ}{\mathbb{Z}}

\newtheorem{theorem}{Theorem}[section]
\newtheorem{lemma}[theorem]{Lemma}
\newtheorem{corollary}[theorem]{Corollary}
\newtheorem{proposition}[theorem]{Proposition}

\theoremstyle{definition}

\newtheorem{remark}[theorem]{Remark}

\theoremstyle{definition}

\theoremstyle{definition}

\theoremstyle{definition}
\newtheorem{example}[theorem]{Example}

\begin{document}



 \title[]{On the flatness of the Bergman space
   \\as a module over the Hardy algebra}

\author[]{Amol Sasane}
\address{Mathematics Department, London School of Economics, 
Houghton Street, London WC2A 2AE, UK}
\email{A.J.Sasane@LSE.ac.uk}
 
\maketitle

\begin{abstract} 
Let $A^2$ denote the Bergman space of the open unit disc $\mathbb{D}$ in
the complex plane, consisting of holomorphic functions that are square-integrable with respect to the area measure in the disc. Let 
$H^\infty$ denote the Hardy algebra, consisting of bounded and holomorphic functions
on $\mathbb{D}$.  Equipped with pointwise operations, $A^2$ is an
$H^\infty$-module.  It is shown that the $H^\infty$-module $A^2$ is
not flat. It is also shown that there exists a maximal ideal
$\mathfrak{m}$ in $H^\infty$ such that $\mathfrak{m}A^2\!=\!A^2$.
\end{abstract}

\section{Introduction}

\noindent The aim of this article is to show that the
$H^\infty$-module $A^2$ is not flat, and moreover to prove that there
exists a maximal ideal $\mathfrak{m}$ in $H^\infty$ such that
$\mathfrak{m}A^2\!=\!A^2$.

Let $R$ be a commutative unital ring.  The notions of flatness and
faithful flatness of $R$-modules were introduced by
Jean-Pierre~Serre in his influential `GAGA' paper \cite{Ser}. Recall
that an $R$-module $M$ is {\em flat} if the functor\footnote{We recall that the
functor $-\otimes_{R} M$ acts as follows: Objects, namely $R$-modules
$N$, are sent to $N\otimes_R M$, and morphisms, which are $R$-linear
maps $f:N\to N'$ between $R$-modules $N,N'$, are sent to morphisms
$f\otimes 1_M: N\otimes_R M\to N'\otimes_R M$, mapping elements
$n\otimes m\in N\otimes_R M$ to $f(n)\otimes m\in N'\otimes_R M$.}
$-\otimes_{R} M$, from the category $R$-Mod of $R$-modules to $R$-Mod,
is exact. That is, for every short sequence
\begin{equation}
\textstyle
0\longrightarrow N'\longrightarrow N \longrightarrow N''\longrightarrow 0,
\tag{$S$}
\end{equation}
which happens to be exact, we have that the new short sequence 
\begin{equation}
\textstyle 
0\longrightarrow N'\otimes_{R} M\longrightarrow N \otimes_{R} M
\longrightarrow N''\otimes_{R} M\longrightarrow 0
\tag{$\widetilde{S}$}
\end{equation}
is exact. Flatness of modules is a subtle property,
as illustrated by the following examples.

\begin{example}(Case when $R$ is a principal ideal domain.) 

\noindent 
If $R$ is a principal ideal domain, then an $R$-module $M$ is flat if
and only if whenever $r\!\in\! R$ and $m\!\in\! M$ satisfy\footnote{We recall
that such an $m$ is called a {\em torsion} element. Thus, when $R$ is
a principal ideal domain, flatness is equivalent to 
torsion-freeness. Geometrically, the terminology `flat' reflects the intuition
that the quasi-coherent sheaf associated to $M$ varies uniformly over
$\text{Spec}\;R$, having well-behaved fibers (see
\cite[Chapter~24]{Vak}).}  $r\!\cdot\! m\!=\!0$, we have $r\!=\!0$ or $m\!=\!0$ (see,
e.g., \cite[\S4.2, Proposition~7, p.119]{Bos}).

Thus the $\mZ$-module $\mZ/2\mZ$ (integers modulo $2$) is not flat
(since for example $2\!\cdot\! [1]\!=\![2]\!=\![0]$ in $\mZ/2\mZ$).

Every complex vector space, regarded as  a
$\mC$-module, is flat.

If $R\!=\!\mC[z]$ (set of all complex polynomials viewed as functions on $\mC$),
$M\!=\!\calO(\mC)$ (set of all entire functions), and all operations are
defined in a pointwise manner, then the $R$-module $M$ is flat.
\hfill$\Diamond$
\end{example}

\begin{example}(Free $\subset$ Projective $\subset$ Flat.)

\noindent Recall that an $R$-module $M$ is {\em free} if it has a
basis. Equivalently, there exists an index set $I$ such that $M$ is
isomorphic to the $R$-module $R^{(I)}$ of finitely supported
maps from $I$ to $R$, equipped with pointwise operations). An $R$-module $M$ is {\em projective} if it is the direct summand of a free module, that is, 
there exists an $R$-module $N$ and an index set $I$ such that $M\oplus N\!\cong \!R^{(I)}$. Every free module is
projective, and every projective module is flat (see,
e.g., \cite[\S4.2, p.123]{Bos}).

If $\calM$ is a smooth manifold, $R\!=\!C^\infty(\calM)$ is the ring of all real-valued
smooth functions on $\calM$, $M\!=\!\chi(\calM)$ is the $R$-module of all vector fields on
$\calM$, with the usual module action  (given by $(fX)_p\!=\!f(p) X_p$ for all $p\in \calM$,
$f\in C^\infty(\calM)$ and $X\in \chi(\calM)$), then the $R$-module
$M$ is flat. Indeed, the `smooth' Serre-Swan theorem
(see \cite[\S12.33]{Nes}) states that the $C^\infty(\calM)$-module $
\chi(\calM)$ is a finitely-generated projective module, and hence it is flat.
\hfill$\Diamond$
\end{example}

\begin{example}(A matricial characterisation of flatness.)

\noindent 
An $R$-module $M$ is flat if and only if for all $n\in \mN$, for all
$r_1,\cdots,r_n\in R$ and for all $m_1,\cdots,m_n\in M$ satisfying the
relation
\vspace{-0.12cm}
$$
\textstyle \sum\limits_{i=1}^n r_i m_i = 0,
$$
there exist $k\in \mN$, $\theta_{ij}\in R$ ($1\!\le\! i\!\le\! n$,
$1\!\le\! j\!\le \!k$), and $\mu_1,\cdots, \mu_k\in M$ such that
\vspace{-0.15cm}
$$
\textstyle
\begin{array}{ll}
  m_i = \sum\limits_{j=1}^k \theta_{ij} \mu_j
  &\!\!\text{for all }i\in \{1,\cdots, n\}, 
\text{ and }\\[-0.09cm]
\textstyle \sum\limits_{i=1}^n r_i\theta_{ij}=0
&\!\!\text{for all }j\in \{1,\cdots,k\} 
\end{array}
$$
(see \cite[Chapter~I, \S2, Proposition~13]{Bou}). 

\noindent If $R\!=\!C_{\text{b}}(\mR)$, the ring of all complex-valued continuous
and bounded functions on $\mR$, and $M\!=\!\textstyle C_{\text{poly}}(\mR)$ is the set of all continuous functions with at most polynomial growth, that is, 
$$
\textstyle 
\!C_{\text{poly}}(\mR)
\!:=\!\left\{f:\mR\to \mC\;\Big|
  {\scaleobj{0.96}{\begin{array}{ll}
        f\text{ is continuous, }\exists M_f\!>\!0\text{ and }
        \exists k_f\in \mN\cup\{0\}\\
        \text{such that }\forall x\in \mR, \,
        |f(x)|\!\le\! M_f(1+x^2)^{k_f}
  \end{array}}}
  \!\!\right\},
$$
and all operations are defined in a pointwise manner, then  the
above characterisation shows that the $R$-module $M$ is
flat.

It was shown in \cite{Sas} that, if $X$ is a locally compact
Hausdorff space equipped with a nonnegative Radon measure $\mu$,
$R\!=\!L^\infty(X,\mu)$ (set of all essentially bounded complex-valued
functions on $X$), $M\!=\!L^2(X,\mu)$ (set of all square-integrable
complex-valued functions on $X$), and all operations are defined in a
pointwise manner, then the $R$-module $M$ is flat.

Let $ {\mathbb{D}}\!=\!\{z\in {\mathbb{C}}\!:\!|z|\!<\!1\}$.
Denote by
${\mathcal{O}}({\mathbb{D}})$ the set of all holomorphic functions on $\mD$. The {\em Hardy algebra} $H^\infty$ is
the Banach algebra of bounded functions $f\!\in\!
{\mathcal{O}}({\mathbb{D}})$, equipped with pointwise operations, and
the supremum norm ($\|f\|_\infty\!:=\!\sup_{z\in {\mathbb{D}}}|f(z)|$
for $f\in H^\infty$).  The {\em Hardy space} $H^2$ is the Hilbert
space of all $f\in {\mathcal{O}}({\mathbb{D}})$ such that
$$
\textstyle \|f\|_{H^2}^2\!:=\!\sum_{n=0}^\infty |c_n|^2\!<\!\infty,
\text{ where }f(z)\!=\!\sum_{n=0}^\infty c_n z^n\,\, (z\in {\mathbb{D}}).
$$
With pointwise multiplication, $H^2$ is an $H^\infty$-module. 
The vector-valued Beurling theorem (see, e.g., \cite[Corollary~6, p.17-18]{Nik}),  
characterising, 
for each $n \in \mathbb{N}$, 
 the closed (componentwise) shift-invariant subspaces of
the direct sum $(H^2)^n := H^2 \oplus \cdots \oplus H^2$ ($n$ copies), implies that
$H^2$ is a flat $H^\infty$-module (see \cite[Proposition~3.1]{Sas}).
\hfill$\Diamond$
\end{example}

\noindent Next we turn our attention to faithful flatness.
An $R$-module $M$ is {\em faithfully flat}
if for every short sequence of $R$-modules 
\begin{equation}
\textstyle
0\longrightarrow N'\longrightarrow N \longrightarrow N''\longrightarrow 0,
 \tag{$S$}
\end{equation}
the sequence $(S)$ is exact if and only if the sequence 
\begin{equation}
\textstyle 
0\longrightarrow N'\otimes_{R} M\longrightarrow N \otimes_{R} M
\longrightarrow N''\otimes_{R} M\longrightarrow 0
\tag{$\widetilde{S}$}
\end{equation}
is exact. For a flat $R$-module $M$,  faithful flatness is  equivalent to the condition that
$\mathfrak{m} M\subsetneq M$ 
 for every
maximal ideal $\mathfrak{m}$ of $R$. Here,
$$
\textstyle \mathfrak{m}M\!:=\!
\Big\{\;
    {\scaleobj{0.96}{\sum\limits_{i=1}^n r_i m_i:
        n\!\in\! \mathbb{N}, \text{ and for all }i\in\! \{1,\cdots,n\},
        \;r_i\!\in\! \mathfrak{m} \text{ and } m_i\!\in\! M}}
\Big\}.
$$
This motivates the study of the opposite algebraic situation,
namely when there exists a maximal ideal $\mathfrak{m}$ of $R$ such that
$\mathfrak{m}M\!=\!M$. If $M$ is finitely-generated, then by Nakayama's lemma (see, e.g.,
\cite[\S8.2]{Vak}), this condition is equivalent to
$M_{\mathfrak{m}}\!=\!0$, or equivalently,
$\mathfrak{m}\not\in \operatorname{supp}M$.
Thus, in the geometric language of sheaf theory\footnote{See,
for example, \cite[\S2.7]{Vak} for the notion of the support of a
sheaf, \cite[\S4.1]{Vak} for the construction of the quasi-coherent
sheaf $\widetilde{M}$ associated to an $R$-module $M$,
\phantom{[12]} \cite[Exercise~4.1.E]{Vak} for the identification
$(\widetilde{M})_{\mathfrak{m}}\simeq M_{\mathfrak{m}}$, and
\cite[Definition~4.1.7]{Vak} for the definition of the support of a module.\phantom{$(\widetilde{M})_{\mathfrak{m}}$}}, it means that the 
associated quasi-coherent sheaf vanishes in a neighbourhood of $\mathfrak{m}$. 

We list some examples below.

We have seen that the $\mZ$-module $\mZ/2\mZ$ is not flat, and hence it cannot be faithfully flat. Nevertheless,  
$\mathfrak{m}(\mZ/2\mZ)\!=\!\mZ/2\mZ$ for some 
 maximal ideal $\mathfrak{m}$ of $\mZ$. Indeed, the maximal ideals of $\mZ$ are the
principal ideals $\langle p\rangle$ generated by primes $p$. If
$p\!\neq \!2$,  then $p$ is odd, and $p\!\equiv\! 1\!\! \pmod 2$. Therefore  $p\!\cdot\! [1]\!=\![1]$ and $p\!\cdot\! [0]\!=\![0]$, and since $[0]$ and $[1]$ are the only elements of $\mZ/2\mZ$, we obtain $\mathfrak{m}(\mZ/2\mZ)\!=\!\mZ/2\mZ$, where $\mathfrak{m}\!=\!\langle p\rangle$.

On the other hand, consider the $\mathbb Z$-module $
M\!=\!\mZ\oplus(\mZ/2\mZ)$. This module is not flat, 
as there exists a torsion element: $2\cdot (0,[1])\!=\!(0,[0])$. 
However, for every prime $p$, 
$\langle p\rangle M
\!=\!
\langle p\rangle\oplus
\langle p\rangle(\mZ/2\mZ)
\!\subsetneq\!
\mZ\oplus(\mZ/2\mZ),
$ 
and hence there is no maximal ideal $\mathfrak{m}$ of $\mZ$ such
that $\mathfrak{m}M\!=\!M$.

If $V$ is a nonzero complex vector space, regarded as a
$\mC$-module, then the only maximal ideal in the field $\mC$ is
$\mathfrak{m}\!=\!\{0\}$. Thus 
$\mathfrak{m} V\!=\!\{0\}\!\subsetneq \!V$.

If $R\!=\!\mC[z]$ and $M\!=\!\calO(\mC)$, then  for each maximal ideal
$\mathfrak{m}$ of $\mC[z]$, we have 
$\mathfrak{m}\calO(\mC)\!\subsetneq \!\calO(\mC)$. Indeed, 
by the Hilbert nullstellensatz, every maximal ideal of $\mC[z]$ has the form
$\langle z\!-\!\lambda\rangle\!:=\!\{(z\!-\!\lambda)p\!:\!p\in \mC[z]\}$
for some $\lambda \in \mC$. Every function in
$\langle z-\lambda\rangle \calO(\mC)$ vanishes at $\lambda$, but the constant function $1\in \mathcal{O}(\mC)$ does not.

If $\calM$ is a smooth manifold, $R\!=\!C^\infty(\calM)$, $M\!=\!\chi(\calM)$,
then for every maximal ideal $\mathfrak{m}$ of $C^\infty(\calM)$, we have 
$\mathfrak{m}\,\chi(\calM)\subsetneq \chi(\calM)$. This follows from the fact that $\chi(\mathcal{M})$ is a finitely-generated projective $C^\infty(\mathcal{M})$-module. 
Indeed, by Nakayama's lemma, it suffices to show that 
$$
\textstyle 
\operatorname{Ann}_{C^\infty(\mathcal{M})}(\chi(\mathcal{M}))\!=\!\{0\},
$$ 
where
 $
\textstyle 
\operatorname{Ann}_{C^\infty(\mathcal{M})}(\chi(\mathcal{M}))
\!:=\!
\{f\in C^\infty(\mathcal{M}):fX\!=\!0\text{ for all }X\in\chi(\mathcal{M})\}.
$ 
Let $f\in C^\infty(\mathcal{M})$ be not identically zero. Then there exists
$p\in\mathcal{M}$ such that $f(p)\neq0$. Choose a chart
$(U,\mathbf{x})$ containing $p$, and let
$\varphi\in C^\infty(\mathcal{M})$ be a cut-off/bump function satisfying
$\varphi\!\equiv\! 1$ in a neighbourhood $V$ of $p$ and
$\varphi\!\equiv\! 0$ outside a neighbourhood $W$ of $p$, with
$V\!\subset\!\overline{W}\!\subset\! U$. 
Denote the first component of the chart map
$\mathbf{x}:U\to \mR^d$ by $x_1:U\to \mR$. Let 
${\scaleobj{0.93}{\frac{\partial\hfill }{\partial x_1}}}\in \chi(U)$
be the corresponding chart-induced local vector field. Then
 $\varphi x_1\in C^\infty(\calM)$, $X\!:=\!\varphi {\scaleobj{0.93}{\frac{\partial\hfill }{\partial  x_1}}}\in \chi(\calM)$,
and we have

\vspace{0.06cm}

$
\textstyle 
\begin{array}{rcl}
((f X)(\varphi x_1))(p)\!=\!f(p)X_p(\varphi x_1)
  \!\!\!\!\!&=&\!\!\!\!
  f(p)X_p(x_1)
  \!=\!f(p)\varphi(p)(\frac{\partial\hfill }{\partial x_1})_p x_1
\\[0.06cm]
\!\!\!\!\!&=&\!\!\!\!
f(p) \!\cdot\! 1\!\cdot\! 1\!\neq\! 0.
\end{array}
$

\vspace{0.06cm}

\noindent 
Hence $fX\!\neq \!0$, and therefore $\textstyle \operatorname{Ann}_{C^\infty(\mathcal{M})}(\chi(\mathcal{M}))\!=\!\{0\}$.  

\goodbreak 

\noindent 
If $R\!=\!C_{\text{b}}(\mR)$ and $M\!=\!\textstyle C_{\text{poly}}(\mR)$, then
there exists a maximal ideal $\mathfrak{m}$ of $C_{\text{b}}(\mR)$
such that $\mathfrak{m} C_{\text{poly}}(\mR)\!=\!C_{\text{poly}}(\mR)$. Indeed, consider the ideal 
$$
\textstyle 
\mathfrak{i}:=\Big\{g\in C_{\text{b}}(\mR):
\lim\limits_{x\to \infty} g(x)\text{ and }\lim\limits_{x\to -\infty} g(x)\text{ exist and equal }0\Big\}
$$ 
of $C_{\text{b}}(\mR)$. It is a proper ideal of $C_{\text{b}}(\mR)$, and therefore  is contained in
some maximal ideal $\mathfrak{m}$. Now let 
$f\in C_{\text{poly}}(\mR)$. Then we can write $f\!=\!gh$, where
$g\!:=\!\frac{1}{1+x^2}\in \mathfrak{i}\subset \mathfrak{m}$ and
$h\!:=\!(1+x^2) f\in C_{\text{poly}}(\mR)$. Since $f\in C_{\text{poly}}(\mR)$  was arbitrary, we conclude that  $\mathfrak{m} C_{\text{poly}}(\mR)\!=\!C_{\text{poly}}(\mR)$. 

The unfaithfulness of some flat modules arising in analysis was
investigated in \cite{Sas}. In particular, the following results
were obtained:
\begin{itemize}
\item[${\scaleobj{0.81}{\bullet}}$]
  $\pmb{(}R\!=\!L^\infty(X,\mu)$ and $M\!=\!L^2(X,\mu).\pmb{)}$
  
  \noindent 
  Suppose that $X$ is a locally compact Hausdorff topological space,
  equipped with a  positive Radon measure $\mu$. Let $X\!=\!\bigcup_{n=1}^\infty U_n$, 
  where $(U_n)_{n\in \mN}$  is an increasing sequence of Borel sets $U_n$,
   such that $\overline{U_n}$ is compact and $\mu(U_{n+1}\setminus U_n)\!>\!0$ for every  
   $n\in {\mathbb{N}}$. Then
  there exists a maximal ideal $\mathfrak{m}$ of $L^\infty(X,\mu)$
  satisfying $\mathfrak{m} L^2(X,\mu)\!=\!L^2(X,\mu)$. This applies, for example, when 
   $X\!=\!\mN$ with the counting measure (so that $R\!=\!\ell^\infty$,
  $M\!=\!\ell^2$), and when $X\!=\!\mT\!:=\!\{z\in \mC\!:\!|z|\!=\!1\}$ with the Lebesgue
  measure (so that $R\!=\!L^\infty(\mT)$ and $M\!=\!L^2(\mT)$).

\vspace{0.12cm} 

\item[${\scaleobj{0.81}{\bullet}}$] $\pmb{(}R\!=\!H^\infty$ and $M\!=\!H^2.\pmb{)}$ 
  
  \noindent There exists a maximal ideal $\mathfrak{m}$ of $H^\infty$ such that
  $\mathfrak{m} H^2\!=\!H^2$.
\end{itemize}
Recall that the Bergman space $A^2$ is the Hilbert space of all $f\in
\mathcal{O}(\mD)$ such that
$$
\textstyle 
\|f\|_{A^2}:=\sqrt{\int_{\mD}|f(z)|^2\,dA(z)}<\infty,
$$
where $dA$ is the area measure on $\mD$. 
Equipped with pointwise multiplication, $A^2$ is  an $H^\infty$-module, since 
if $g\in H^\infty$ and $f\in A^2$, then
$$
\textstyle \|g\!\cdot\! f\|_{A^2}\le \|g\|_\infty\|f\|_{A^2}.
$$ 
In light of the known result that the flat $H^\infty$-module $H^2$ 
 is not faithfully flat, it is natural to ask whether the same phenomena occur for the $H^\infty$-module $A^2$. In this article, we answer the following questions:
\begin{itemize}
\item[${\scaleobj{0.81}{\bullet}}$]
  Is  the $H^\infty$-module $A^2$ flat? 

\noindent No (Theorem~\ref{thm_1}).

\vspace{0.1cm}

\item[${\scaleobj{0.81}{\bullet}}$]
  Does there exist a maximal ideal $\mathfrak{m}$ of $H^\infty$ such that 
$\mathfrak{m} A^2\!=\!A^2$? 

\noindent Yes (Theorem~\ref{thm_2}).
\end{itemize} 
Theorem~\ref{thm_1}  is proved in Section~\ref{section_not_flat}, and Theorem~\ref{thm_2}
is proved in Section~\ref{section_not_faithful}. The
examples in the introduction are summarised in the following table.
\begin{center}
{\scaleobj{0.93}{
\begin{tabular}{|c|c|c|c|c|}
\hline
\makecell{$R$} &
\makecell{$M$} &
\makecell{Flat?} &
\makecell{Faithfully flat?} &
\makecell{$\exists$ maximal $\mathfrak{m}\subset R$\\
such that $\mathfrak{m}M\!=\!M$?}
\\ \hline \hline
$\mZ$ & $\mZ/2\mZ$ &  No  & Not applicable & Yes \\ \hline
$\mZ$ & $\mZ\oplus (\mZ/2\mZ)$ &  No  & Not applicable & No \\ \hline  
$\mC$ & $V$ & Yes & Yes & No \\ \hline
$\mC[z]$ & $\calO(\mC)$ & Yes & Yes & No\\ \hline
$C^\infty(\calM)$ & $\chi(M)$ & Yes & Yes & No \\ \hline
$C_{\text{b}}(\mR)$ & $C_{\text{poly}}(\mR)$ & Yes & No & Yes \\ \hline
$L^\infty(X,\mu)$ & $L^2(X,\mu)$ & Yes & No & Yes \\ \hline 
$H^\infty$ & $H^2$ & Yes & No & Yes \\ \hline
 $H^\infty$ & $A^2$ & No & Not applicable & Yes \\ \hline
\end{tabular} }}
\end{center}

\goodbreak
 
\section{The $H^\infty$-module $A^2$ is not flat} 
\label{section_not_flat}
 
\noindent By the matricial characterisation of flatness recalled in the
introduction, given $r_1,\cdots,r_n\in H^\infty$ and $f_1,\cdots,f_n\in A^2$
satisfying the relation
$$
\textstyle \sum\limits_{i=1}^n r_i f_i = 0,
$$
we ask whether there always exist $k\in\mathbb N$,
$\theta_{ij}\in H^\infty$ ($1\!\le\! i\!\le\! n$, $1\!\le\! j\!\le\! k$), and
$\mu_1,\cdots,\mu_k\in A^2$ such that
$$
\textstyle
\begin{array}{ll}
f_i = \sum\limits_{j=1}^k \theta_{ij} \mu_j
&\!\! \text{for all }i\in \{1,\cdots,n\}, 
\text{ and }\\[0.12cm]
\sum\limits_{i=1}^n r_i\theta_{ij}=0
&\!\! \text{for all }j\in \{1,\cdots,k\}.
\end{array}
$$
We will show that this fails already for $n\!=\!2$.

\subsection{An algebraic reduction}$\;$
\label{subsection_algebraic_reduction}

\noindent For $a\!\in\!\mD$, define the disc automorphism $b_a$ by 
$$
\textstyle
b_a(z) := \frac{a-z}{1-\bar az}\,\text{ for all }z\in \mD.
$$
Then $b_a$ is an involution: $b_a\circ b_a \!=\! \mathrm{id}_\mD$. For $a\in \mD$, set 
$$
\textstyle 
\epsilon_a:=
\bigg\{ \begin{array}{cl} 
\frac{|a|}{a}&\text{if }a\!\neq\! 0,\\[0.1cm]
1 &\text{if }a\!=\!0.
\end{array}
$$ 
Then $\epsilon_a\in \mT$ for all $a\in \mD$, and we recall that this unimodular constant is the normalising factor used in the definition of a Blaschke product (see, e.g., \cite[Lemma, p.64]{Hof}, \cite[Chapter~II, Theorem~2.2]{Gar}): it does not affect the modulus of $\epsilon_a b_a$, since $|\epsilon_a b_a(z)|=|b_a(z)|$ for all $z\in \mD$, but it is needed  to guarantee convergence of an infinite product of such factors. 

\begin{lemma}[Division in $H^\infty$ by a Blaschke product]$\;$
\label{lem:boundedquotient}

\noindent 
Let $(a_n)_{n\in \mN}$  in $\mD$ and $(m_n)_{n\in \mN}$  in $\mN$ satisfy
 $$
\textstyle \sum\limits_{n=1}^\infty  m_n(1-|a_n|)<\infty. 
$$
\noindent Let $r$ be the Blaschke product
 $$
 \textstyle r:=\prod\limits_{n=1}^\infty (\epsilon_{a_n} b_{a_n})^{m_n}.
$$
\noindent  
If  $f\in H^\infty$ is such that $\frac{f}{r}$
extends to a holomorphic function on $\mD,$ 
 
\noindent then $\frac{f}{r}\in H^\infty$
and $\|\frac{f}{r}\|_\infty\le\|f\|_\infty$.
\end{lemma}
\begin{proof}
Set $g\!:=\!\frac{f}{r}$. By assumption, $g\in \calO(\mD)$. We have $f\!=\!rg$.
For each $N\in \mN$, write $r\!=\!r_Ns_N$, where
$$
\textstyle 
\begin{array}{rcl}
r_N(z)\!\!\!\!\!&:=&\prod\limits_{n=1}^N\;(\epsilon_{a_n}b_{a_n}(z))^{m_n} 
\; \text{ (finite Blaschke product), and}\\[0.33cm]
\textstyle
s_N(z)\!\!\!\!\!&:=&\!\!\!\prod\limits_{n=N+1}^\infty \!\!(\epsilon_{a_n}b_{a_n}(z))^{m_n} \;\;
\text{(tail)}.
\end{array}
$$
Since 
$$
\textstyle \sum_{n=N+1}^{\!\!\!\infty} m_n(1-|a_n|)<\infty,
$$
 $s_N$ converges uniformly on compact subsets of $\mD$ (see, for example,  \cite[Lemma,~p.64]{Hof}), and defines a holomorphic function there. Hence $f\!=\!r_N(s_Ng)$, where $s_Ng$ is 
holomorphic on $\mD$. Thus  $ g_N\!:=\!\frac{f}{r_N}\!=\!s_Ng$ is
holomorphic on $\mD$.

We claim that $g_N\in H^\infty$. Since $r_N$ is a finite Blaschke
product, it is a rational function whose only poles lie at the points
$1/\bar a_n$ ($1\!\le\! n\!\le\! N$), each of modulus $>\!1$. 
Let $\delta\!>\!0$ and $\delta<1-\max\{|a_n|:1\!\le\! n\!\le\! N\}$. Then $r_N$ has no zero in the closed annulus $\overline
A_\delta\!:=\!\{z\in \mC:1-\delta\!\le\!|z|\!\le\!1\}$.
Since $r_N$ is continuous and nonvanishing on the compact set $\overline A_\delta$,
$$
\textstyle 
m_\delta:=\min\limits_{z\in\overline A_\delta}|r_N(z)| > 0.
$$
Hence
$$
\textstyle
|g_N(z)|=\frac{|f(z)|}{|r_N(z)|}\le\frac{\|f\|_\infty}{m_\delta}
\text{ for all }z\in\overline A_\delta, 
$$
while on the compact disc $\{z\in \mC: |z|\le1-\delta\}$,
$g_N$ is continuous, and hence bounded.
Thus $g_N$ is bounded on $\mD$, i.e.\ $g_N\in H^\infty$.

Since $|r_N(\zeta)|\!=\!1$ for all $\zeta\in \mT$ (as $r_N$ is a finite
Blaschke product), the radial boundary values of $g_N$, which exist
almost everywhere, satisfy
$$
\textstyle 
|g_N(\zeta)|=\frac{|f(\zeta)|}{|r_N(\zeta)|}
=\frac{|f(\zeta)|}{1}
=|f(\zeta)| \le \|f\|_\infty\text{ for almost all }\zeta\in \mT.
$$
Since $g_N\in H^\infty$, we get $\|g_N\|_\infty \le \|f\|_\infty$
for all $N\in \mN$.

Fix $z\in\mD$. Then
$|s_N(z)g(z)|\!=\!|g_N(z)|\!
\le\! \|g_N\|_\infty
\!\le\! \|f\|_\infty$ for all $N\in \mN$.
Since the tail satisfies $s_N(z)\to 1$ as $N\to\infty$, we get
$|g(z)|\le\|f\|_\infty$. As $z\in\mD$ was arbitrary, $g\in H^\infty$
and $ \|g\|_\infty\le\|f\|_\infty$.
 \end{proof}

\begin{proposition}[Free relation module of rank $1$]$\;$
\label{prop:cyclic}

\noindent 
Suppose $r_1,r_2$ are Blaschke products with disjoint zero sets
$Z_1,Z_2\subset\mD,$  respectively. Then
$$
\textstyle
\begin{array}{rcl}
  M \!\!\!\!&:=&\!\!\!\!
  \{ (\theta_1,\theta_2)\in H^\infty\times H^\infty :
  r_1\theta_1+r_2\theta_2 = 0 \} \\[0.1cm]
\!\!\!\!&=&\!\!\!\!  \{ h\!\cdot \!(-r_2,r_1): h\in H^\infty\}.
\end{array}
$$
\end{proposition}
\begin{proof} 
The inclusion $ \{ h\!\cdot\! (-r_2,r_1): h\in H^\infty\}\subset M$ is
immediate, since $r_1(-hr_2)+r_2(hr_1)\!=\!0$. We prove the reverse inclusion.  Suppose that 
$(\theta_1,\theta_2)\in M\setminus \{(0,0)\}$. 
Then $r_1\theta_1 \!=\!-r_2 \theta_2$, and $\theta_1$ is not identically $0$ (otherwise $r_2\theta_2\!=\!-r_1\theta_1\equiv 0$ would imply, thanks to the fact that $r_2$ is not identically $0$, that $\theta_2\!\equiv\! 0$). Similarly, $\theta_2$ is not identically $0$. 
Let $\zeta\in Z_1$,
and let $m$ be its order as a zero of $r_1$.
As $Z_1\cap Z_2\!=\!\emptyset$, we have $r_2(\zeta)\!\neq\! 0$.
Upon comparing the order of the zero $\zeta$ on both sides of 
$r_1\theta_1 \!=\! -r_2 \theta_2$, we conclude that the right-hand side
has $\zeta$ as a zero with at least order $m$.
As $r_2(\zeta)\!\neq\! 0$, it follows that $\theta_2(\zeta)\!=\!0$,
and the order of $\zeta$ as a zero of $\theta_2$ is at least $m$.
Thus $h_1\!:=\!\theta_2/r_1$ has a holomorphic extension to $\mD$.
By Lemma~\ref{lem:boundedquotient},
we have $h_1\!=\!\theta_2/r_1\in H^\infty$.
Similarly, $h_2\!:=\!\theta_1/r_2\in H^\infty$.
Using $r_1\theta_1 \!=\! -r_2 \theta_2$, we get
$r_1(h_2r_2) \!= \!-r_2 (h_1r_1)$ on $\mD$.
As $r_1r_2$ is not identically $0$ in
$\mD$, it follows that $h_2\!=\!-h_1$. Hence $(\theta_1,\theta_2) \!=\! h_1\!\cdot\!(-r_2,r_1)$,
as wanted.
\end{proof}

\begin{corollary}[Collapse of $k$]
\label{cor:collapse}
 Suppose $r_1,r_2$ are Blaschke products with disjoint zero sets $Z_1,Z_2\subset\mD,$  respectively. Then for any  $k\in \mN,$ any $\theta_{ij}\in H^\infty,$ $i\in \{1, 2\},$ $j\in \{1,\cdots ,k\},$  satisfying 
$$
\textstyle 
 r_1\theta_{1j}+ r_2\theta_{2j}=0\text{ \em for  all }j\in \{1,\cdots, k\},
$$ 
and any $\mu_1,\cdots,\mu_k\in A^2,$ there exists a $g\in A^2$ such that
$$
\textstyle 
\sum\limits_{j=1}^k \theta_{1j}\mu_j = -gr_2 \text{ \em and }
\sum\limits_{j=1}^k \theta_{2j}\mu_j = gr_1.
$$
\end{corollary}
\begin{proof}
By Proposition~\ref{prop:cyclic}, there exist $h_1,\cdots, h_k\in H^\infty$ 
such that 
$$
\textstyle (\theta_{1j}, \theta_{2j})\!=\!h_j\!\cdot\! (-r_2,r_1) \text{ for all }
j\in \{1,\cdots, k\}.
$$ 
 Since multiplication by an $H^\infty$ element defines a bounded operator on $A^2$, we have that 
$$
\textstyle g:=\sum\limits_{j=1}^k h_j\mu_j\in A^2.
$$

\goodbreak 

\noindent 
Finally, 
\vspace{-0.42cm}
$$
\textstyle 
\begin{array}{rcl}
\sum\limits_{j=1}^k \theta_{1j}\mu_j 
\!\!\!&=&\!\!\!\sum\limits_{j=1}^k h_j (-r_2) \mu_j =-r_2 g,\quad \text{ and }\\[0.15cm]
 \sum\limits_{j=1}^k \theta_{2j}\mu_j 
\!\!\!&=&\!\!\!\sum\limits_{j=1}^k h_j r_1 \mu_j =r_1 g.
\end{array}
$$
This completes the proof. 
\end{proof}

\begin{proposition}[Reduction to a single function]$\;$
\label{prop:bijection}

\noindent Suppose $r_1,r_2$ are Blaschke products with disjoint zero sets $Z_1,Z_2\subset\mD,$ respectively. Set
 $$
\textstyle 
\begin{array}{rcl}
N \!\!\!&:=&\!\!\! \{ (f_1,f_2)\in A^2\times A^2 : r_1f_1+r_2f_2 = 0 \},\text{ and} \\[0.06cm]
G\!\!\!&:=&\!\!\!\{ g \in \calO(\mD ) : r_1 g, r_2 g \in A^2 \}.
\end{array}
$$  
Then  the map $G\owns g \mapsto (r_2 g,\, -r_1 g)$ is a bijection from $G$ 
onto $N$.
\end{proposition}
\begin{proof}
If $g\in G$, then  $r_1(r_2g)+r_2(-r_1g)\!=\!0$, and so $(r_2g, -r_1 g)\in N$. Thus the map is well-defined. 

\smallskip 

\noindent (Surjectivity:) If $(f_1,f_2)\!=\!(0,0)$, then set $g\!=\!0$. Let $(f_1,f_2)\in N\setminus \{(0,0)\}$.  Set $F\!:=\!r_1f_1\!=\!-r_2f_2\in \calO(\mD)$. Then $F$ is not identically $0$ (otherwise $r_1f_1\!=\!-r_2f_2\equiv 0$ would imply
$f_1\!=\!f_2\!\equiv\! 0$). 
If $\zeta\in Z_1$, then the order of $\zeta$ as a zero of $F\!=\!r_1f_1$  is at least  its order as a zero of $r_1$. Similarly, if $\zeta\in Z_2$, the order of $\zeta$ as a zero of $F\!=\!-r_2f_2$  is at least  its order as a zero of $r_2$. As $Z_1\cap Z_2\!=\!\emptyset$, we conclude that $g\!:=\!F/(r_1r_2)$ extends to a holomorphic function on $\mD$. We have  $r_2g\!=\!F/r_1\!=\!f_1\in A^2$, and $-r_1g\!=\!-F/r_2\!=\!f_2\in A^2$. Hence $G\owns g \mapsto (r_2 g,\, -r_1 g)\in N$   is surjective. 

\smallskip

\noindent (Injectivity:) The injectivity of the map $G\owns g \mapsto (r_2 g,\, -r_1 g)$  is an immediate consequence of the identity theorem (since for example,  it is not the case that $r_2\!\equiv\! 0$ in $\mD$).
\end{proof}

\noindent Now suppose we can find Blaschke products $r_1,r_2$ with disjoint zero sets and a holomorphic $g\notin A^2$ with $f_1\!:=\!r_2g\in A^2$ and $f_2\!:=\!-r_1g\in A^2$. 
Then $r_1 f_1+r_2f_2\!=\!0$. Suppose that for some $k\in \mN$, for some $\theta_{ij}\in H^\infty$, and some $\mu_1,\cdots, \mu_k \in A^2$, we have 
$$
\textstyle
f_i = \sum\limits_{j=1}^k \theta_{ij}\mu_j \text{ for all }i\in \{1,2\}, 
\text{ and }
\textstyle \sum\limits_{i=1}^2 r_i\theta_{ij}=0\text{ for all }j\in \{1,\cdots k\}.
$$
Applying Corollary~\ref{cor:collapse},  there exists a $\widetilde{g}\in A^2$ such that
$$
\textstyle 
f_1=\sum\limits_{j=1}^k \theta_{1j}\mu_j = -\widetilde{g}r_2 \text{ and }
f_2=\sum\limits_{j=1}^k \theta_{2j}\mu_j = \widetilde{g}r_1.
$$ 
By the injectivity part of Proposition~\ref{prop:bijection}, $g\!=\!-\widetilde{g}$, and so we get the contradiction that $g\!=\!-\widetilde{g}\in A^2$. This will show the claimed non-flatness of the $H^\infty$-module $A^2$. In the next subsection, we will construct the Blaschke products $r_1,r_2$ with disjoint zero sets and a holomorphic $g\notin A^2$ such that  $f_1\!:=\!r_2g\in A^2$ and $f_2\!:=\!-r_1g\in A^2$.

\subsection{Construction of the Blaschke products $r_1,r_2$ and $g\in \calO(\mD)$}
\label{subsection_construction}

\noindent Before proving Proposition~\ref{lem:main}, we collect some preliminary facts and establish four auxiliary lemmas which will be used in the construction.

 Recall that the {\em pseudohyperbolic distance} between $z,w\in\mD$ is  
 $$
\textstyle
\rho(z,w) := |\frac{z-w}{1-\bar w z}|.
$$ 
 Then $\rho(z,w)\!=\!\rho(w,z)$, and for all  $a\in \mD$,  $|b_a(z)|\!=\!\rho(a,z)$. 
For $w\in\mD$, define the (unnormalised) Bergman kernel function
$$
\textstyle
e_w(z) := \frac{1-|w|^2}{(1-\bar w z)^2} \text{ for all } z\in\mD.
$$
We will often use the fact that for all $z,w\in\mD$, 
\begin{equation}
\label{eqnstar}
\textstyle 
1-\rho(z,w)^2 = \frac{(1-|z|^2)(1-|w|^2)}{|1-\bar w z|^2}. \tag{$\star$}
\end{equation}

\begin{lemma}[Estimate for $\|b_a^{m}\|_{A^2}$]$\;$
\label{lem:A}

\noindent For all $a\in\mD$ and all integers $m\!\ge\! 0,$ 
 $
\textstyle 
\int_\mD |b_a(\zeta)|^{2m}\, dA(\zeta) \le \frac{\pi}{m+1} (\frac{1+|a|}{1-|a|})^2.
$ 
\end{lemma}
\begin{proof}
As $b_a$ is an involutive automorphism of $\mD$, the change of variables $\zeta\!=\!b_a(\eta)$ (so $dA(\zeta)\!=\!|b_a'(\eta)|^2\,dA(\eta)$, where $b_a'(\eta) \!=\! -\frac{1-|a|^2}{(1-\bar a\eta)^{2}}$) yields
$$
\textstyle 
\int_\mD |b_a(\zeta)|^{2m}\,dA(\zeta) =\int_\mD |\eta|^{2m} |b_a'(\eta)|^2\, dA(\eta).
$$
For $\eta\in\mD$, we have $|1-\bar a\eta| \!\ge\! 1-|a|$, and so 
$$
\textstyle |b_a'(\eta)| \le \frac{1-|a|^2}{(1-|a|)^2} = \frac{1+|a|}{1-|a|}.
$$
 The proof is completed noting that $\int_\mD |\eta|^{2m}\,dA(\eta) \!=\! \frac{\pi}{m+1}$ (which is seen easily using polar coordinates).
\end{proof}

\begin{lemma}
\label{lem:B}
For all $w\in\mD$ and all  $\varphi \in H^\infty ,$  
$$
\textstyle
\|\varphi e_w\|_{A^2}^2 = \int_\mD |\varphi(z)|^2 |e_w(z)|^2\, dA(z) = \int_\mD |\varphi(b_w(\eta))|^2 \,dA(\eta).
$$
\end{lemma}
\begin{proof}
We use the change of variables $z\!=\!b_w(\eta)$. As 
$$
\textstyle 1-\bar w\, b_w(\eta) = \frac{1-|w|^2}{1-\bar w\eta}, 
\text{ and }dA(z) = \frac{(1-|w|^2)^2}{|1-\bar w\eta|^4}\, dA(\eta),
$$ 
we obtain
$$
\textstyle 
{\scaleobj{0.96}{
\begin{array}{rcl}
\int_\mD |\varphi(z)|^2 |e_w(z)|^2\,dA(z) 
\!\!\!&=&\!\!\! \int_\mD |\varphi(z)|^2 \frac{(1-|w|^2)^2}{|1-\bar w z|^4}\, dA(z) 
\\[0.3cm]
\!\!\!&=&\!\!\! \int_\mD |\varphi(b_w(\eta))|^2 
\frac{(1-|w|^2)^2}{|1-\bar w b_w(\eta)|^4} \frac{(1-|w|^2)^2}{|1-\bar w\eta|^4}\,dA(\eta)
\\[0.3cm]
\!\!\!&=&\!\!\! \int_\mD |\varphi(b_w(\eta))|^2 
\frac{(1-|w|^2)^2}{\frac{(1-|w|^2)^4}{|1-\bar w \eta|^4}} \frac{(1-|w|^2)^2}{|1-\bar w\eta|^4}\,dA(\eta)
\\[0.36cm]
\!\!\!&=&\!\!\!
\int_\mD |\varphi(b_w(\eta))|^2 \,dA(\eta),
\end{array}}}
$$
as wanted.
\end{proof}

\begin{lemma}[Composition of Blaschke factors]$\;$
\label{lem:C}

\noindent 
For every $a_1,a_2\in\mD,$ there exist  $\alpha\in \mT$ and  $\zeta\in\mD$ with $|\zeta|\!=\!\rho(a_1,a_2)$ such that $
b_{a_1} \!\circ b_{a_2}\!=\! \alpha b_{\zeta}$ in $\mD$.
\end{lemma}
\begin{proof}
Set $\sigma\!:=\!b_{a_1}\!\circ b_{a_2}$. Since $\sigma$ is an automorphism of $\mD$, 
 the characterisation of automorphisms of the disc (see, e.g., \cite[Theorem~13.15]{BakNew}) 
yields  $\sigma\!=\! \alpha b_\zeta$ for some  $\alpha\in \mT$, where $\zeta\!:=\!\sigma^{-1}(0)$. Since both $b_{a_1}$ and $b_{a_2}$ are involutions, $\sigma^{-1}\!=\!b_{a_2}\!\circ b_{a_1}$. Thus $\zeta \!=\! b_{a_2}(b_{a_1}(0)) \!=\! b_{a_2}(a_1)$, and so $|\zeta| \!=\! |b_{a_2}(a_1)|\!=\!\rho(a_1,a_2)$. 
\end{proof}

\begin{lemma}[Reproducing formula]$\;$
\label{lem:D}

\noindent For all $h\in A^2$ and all $w\in\mD,$
 $
\textstyle
\langle h, e_w\rangle_{A^2} = \pi(1-|w|^2)\, h(w).
$
 
\noindent For all $v,w\in\mD,$
 $ 
\textstyle
|\langle e_v,e_w\rangle_{A^2}| = \pi(1-\rho(v,w)^2),
$  
and $\|e_w\|_{A^2}^2 = \pi$.
\end{lemma}
\begin{proof} 
The first claim follows from \cite[Corollary~1.5]{HedKorZhu}.  Setting $h\!=\!e_v$, 
$$
\textstyle \langle e_v,e_w\rangle = \pi(1-|w|^2)e_v(w) = \pi\frac{(1-|w|^2)(1-|v|^2)}{(1-\bar v w)^2}.
$$
 Also,   $| \langle e_v,e_w\rangle|\!=\!\pi\frac{(1-|w|^2)(1-|v|^2)}{|1-\bar v w|^2} \!=\!\pi(1-\rho(v,w)^2)$ by $(\star)$.
Finally, taking $v\!=\!w$, we get  
$
\textstyle \|e_w\|^2_{A^2} \!=\! \pi(1-\rho(w,w)^2)\!=\!\pi.
$ 
\end{proof}

\noindent We are now ready to give the construction, but first we explain the underlying idea.
We take 
$$
\textstyle g=\sum\limits_{k=1}^\infty e_{w_k},
$$
 where $(w_k)_{k\in \mN}$ converges radially to $1\in\mT$, and is chosen 
  so that $\langle e_{w_j},e_{w_k}\rangle$ are sufficiently small for $j\!\neq\! k$. 
 A Bessel inequality then shows that $g\notin A^2$, since $\|e_{w_k}\|^2\!=\!\pi$ for all $k\in \mN$. 
  
To construct the Blaschke products $r_1$ and $r_2$, we exploit  the dependence of   $\|b_a^me_w\|_{A^2}$  on the pseudohyperbolic distance between $a$ and $w$, and on the multiplicity $m$. 
The zero sets of $r_1$ and $r_2$  are  chosen to be disjoint, with each product having, for every $k$, a zero at a fixed pseudohyperbolic distance from $w_k$. 
The zeroes of $r_1$ and $r_2$  are placed on opposite sides of the real axis to ensure  disjointness, and their multiplicities $m_k\to\infty$ as $k\to\infty$.

The key quantitative fact is that if a Blaschke factor has a zero of order $m$ at a  fixed pseudohyperbolic distance from the point $w$, then multiplication by this factor reduces the $A^2$-norm of $e_w$ by a  factor of order $1/\sqrt{m}$. 
Choosing $m_k\!=\!k^3$ then yields 
$$
\textstyle 
\sum\limits_{k=1}^\infty\|r_ie_{w_k}\|_{A^2}<\infty\text{ for }i\in \{1,2\},
$$
 and so $r_1g,r_2g\in A^2$.

\begin{proposition}[Construction]$\;$
\label{lem:main}

\noindent 
There exist Blaschke products $r_1,r_2\in H^\infty$ with disjoint zero sets$,$ and a $g\in \calO(\mD),$ such that $r_1 g \in A^2,$ $r_2 g \in A^2,$ but $g\notin A^2$.
\end{proposition}
\begin{proof} As the proof is long, we split it into a sequence of steps. 

\medskip 

\noindent {\bf Step 1.} (The disjoint sets $Z_1$ and $Z_2$.)

\noindent For each $k\in \mN$, set\footnote{$a_k\!:=\!1-\theta^k$, 
with $\theta\in (0,\frac{1}{9})$, would do: 
See Step~4, where this gives $\eta\!:=\!8\!\cdot\!\frac{\theta}{1-\theta}\!<\!1$.} 
$a_k \!:=\! 1-\frac{1}{10^k}$, and 
$
\textstyle 
\tau_k(\zeta) \!:=\! \frac{\zeta+a_k}{1+a_k\zeta}.
$
Define\footnote{$m_k\!:=\!k^p$, with $p\!>\!2$,  would do: 
See Step 7, where this  gives 
$\sum\limits_{k=1}^\infty\|r_ie_{w_k}\|_{A^2}\!<\!\infty$.}
$$
\textstyle 
w_k \!:=\! \tau_k(0)\! =\! a_k, \quad 
u_k \!:= \!\tau_k(\frac{i}{2}), \quad 
v_k \!:=\! \tau_k(-\frac{i}{2}) \!=\! \overline{u_k}, \quad 
m_k \!:=\! k^3.
$$
 As $\tau_k\!=\!-b_{-a_k}$ is an automorphism of $\mD$, it preserves the hyperbolic distance.  Thus
$$
\textstyle 
\rho(u_k,w_k) \!=\! \rho(\frac{i}{2},0) \!=\! \frac{1}{2} \,\text{ and } \,
\rho(v_k,w_k) \!=\! \rho(-\frac{i}{2},0)\!=\!\frac{1}{2} \,\text{ for all } k\in \mN. 
$$
We have 
 $
\textstyle
u_k=\tau_k(\frac{i}{2})=\frac{i+2a_k}{2+ia_k}
=\frac{(i+2a_k)(2-ia_k)}{(2+ia_k)(2-ia_k)}
= \frac{5a_k+2i(1-a_k^2)}{4+a_k^2},
$ 
 and so 
$$
\textstyle 1-u_k = \frac{(1-a_k)\pmb{(}(4-a_k) - 2i(1+a_k)\pmb{)}}{4+a_k^2}.
$$
Since $a_k\in(0,1)$, the real and imaginary parts of $(4-a_k) - 2i(1+a_k)$   satisfy $(4-a_k)^2\!<\!16$ and $(-2(1+a_k))^2\!<\!16$. So the absolute value of one of the factors in the numerator of the right-hand side above can be estimated above as $|(4-a_k)-2i(1+a_k)| \!< \! \sqrt{16\!+\!16}\!=\!4\sqrt{2}$. Also, the denominator  satisfies $4+a_k^2\!\ge\! 4$. Hence 
\begin{equation}
\textstyle 
1-|u_k| 
\!\le\! |1-u_k| 
\!\le\! \frac{4\sqrt2}{4}(1-a_k) 
\!=\! \sqrt{2}(1-a_k) 
\!\le\! \frac{3}{2}(1-a_k) 
\!= \!\frac{3}{2}\!\cdot\! \frac{1}{10^k}. \tag{$1$}
\end{equation}
As $v_k\!=\!\overline{u_k}$, also $1-|v_k|\!\le\! \frac{3}{2}\!\cdot\! \frac{1}{10^k}$.
Moreover, for all $k\in \mN$, 
$$
\textstyle \text{Im}\;u_k = \frac{2(1-a_k^2)}{4+a_k^2} > 0,\,\text{ and }\,
\text{Im}\;v_k = -\text{Im}\;u_k < 0.
$$
Since $(a_k)_{k\in \mN}$ is strictly increasing, and as the map 
  $$
 \textstyle 
 (0,1)\owns a\mapsto\frac{\frac{i}{2}+a}{1+a\frac{i}{2}}\in \mC
 $$ 
 (which is the restriction of a M\"obius transformation)
 is injective,  the points $u_k$ are pairwise distinct. Thus the $v_k\!=\!\overline{u_k}$ are pairwise distinct too. We therefore obtain
$$
\textstyle 
\begin{array}{ll}
Z_1:=\{u_k:k\in \mN\} \subset \{z\in \mC: \text{Im} \;z>0\}, \\[0.1cm]
Z_2:=\{v_k:k\in \mN\}\subset\{z\in \mC: \text{Im} \;z<0\}, 
 \end{array}
$$
and hence $Z_1\cap Z_2\!=\!\emptyset$.

\medskip

\goodbreak

\noindent {\bf Step 2.}  (Blaschke products $r_1,r_2$  with zero sets $Z_1,Z_2$, respectively.)

\noindent 
As\footnote{The claim follows by induction: It holds for $k\!\in \!\{1,2,3\}$, and if true for some $k\!\ge\! 3$, then 
$(k+1)^3\!=\!k^3(1+\frac{1}{k})^3 \!\le\! k^3 (1+\frac{1}{3})^3 \!=\!k^3 \frac{64}{27}
\!<\!k^3 3\!\le\! 3^k 3\!=\!3^{k+1}$. }
 $k^3\!\le\! 3^k$ for all $k\!\ge\!1$, using ($1$), we get
$$
\textstyle 
m_k(1-|u_k|) \le k^3 \!\cdot\!\frac{3}{2} \!\cdot\! \frac{1}{10^{k}}
=\frac{3}{2} \!\cdot\! \frac{k^3}{3^k}\!\cdot \!(\frac{3}{10})^{k}
 \le \frac{3}{2} \!\cdot \!1 \!\cdot\! (\frac{3}{10})^{k} 
 = \frac{3}{2}(\frac{3}{10})^{k}.
$$
Hence  
$$
\textstyle 
\begin{array}{rcl}
\sum\limits_{k=1}^\infty m_k(1-|u_k|) 
\!\!\!\!&\le&\!\!\!\! \frac{3}{2} \sum\limits_{k=1}^\infty (\frac{3}{10})^{k}
 = \frac{3}{2} \!\cdot\! \frac{\frac{3}{10}}{1-\frac{3}{10}} 
 = \frac{3}{2}\!\cdot\!\frac{3}{7}=\frac{9}{14} <\infty, \text{ and }\\[0.3cm]
\sum\limits_{k=1}^\infty m_k(1-|v_k|)
\!\!\!\!&=&\!\!\!\!\sum\limits_{k=1}^\infty m_k(1-|\overline{u_k}|)
= \sum\limits_{k=1}^\infty m_k(1-|u_k|)\le \frac{9}{14}<\infty.
\end{array}
$$
Thus the Blaschke products 
$$
\textstyle 
r_1 \!:=\!\prod\limits_{k=1}^\infty  (\epsilon_{u_k}b_{u_k})^{m_k} \,\text{ and }\,
r_2 \!:=\! \prod\limits_{k=1}^\infty (\epsilon_{v_k}b_{v_k})^{m_k}
$$
are well-defined. The zero sets of $r_1$ and $r_2$ are $Z_1\!=\!\{u_k:k\in \mN\}$ and $Z_2\!=\!\{v_k:k\in \mN\}$, respectively, and $Z_1\cap Z_2=\emptyset$, as observed in the previous step. 

\medskip

\noindent {\bf Step 3.} (Construction of $g$.)

\noindent We will show 
 $$
\textstyle g := \sum\limits_{k=1}^\infty e_{w_k}
$$  
converges locally uniformly on $\mD$, and hence defines a holomorphic function there.

 For $|z|\!\le\! r\!<\!1$,  
 $
\textstyle |e_{w_k}(z)| \!=\! \frac{1-a_k^2}{|1-a_kz|^2} 
\le  \frac{(1+a_k)(1-a_k)}{(1-|a_k||z|)^2}  \!\le\! \frac{2(1-a_k)}{(1-1\cdot r)^2} \!=\! \frac{2\cdot \frac{1}{10^{k}}}{(1-r)^2}.
$ 
 So
$$
\textstyle 
\sum\limits_{k=1}^\infty \sup\limits_{|z|\le r} |e_{w_k}(z)| 
\le \frac{2}{(1-r)^2} \sum\limits_{k=1}^\infty \frac{1}{10^{k}}
 = \frac{2}{(1-r)^2}\!\cdot\!\frac{1}{9} <\infty.
$$
Hence the series defining $g$ converges uniformly on compact subsets of $\mD$. Hence (see e.g., \cite[Theorem~7.6]{BakNew}) $g$ is holomorphic on $\mD$.

\medskip

\noindent {\bf Step 4.} (An explicit separation bound.)

\noindent Using ($\star$), 
and thanks to $a_k\in (0,1)$, the inequalities $1-a_ja_k \!\ge\! 1-a_j$ and $1-a_k^2 \!=\!(1+a_k)(1-a_k)\!\le\! 2(1-a_k)$, we obtain 
$$
\textstyle
1-\rho(w_j,w_k)^2 \overset{(\star)}{=} \frac{(1-a_j^2)(1-a_k^2)}{(1-a_ja_k)^2} \le \frac{2(1-a_j)\cdot2(1-a_k)}{(1-a_j)^2} = \frac{4(1-a_k)}{1-a_j} = 4\!\cdot \!\frac{1}{10^{k-j}}.
$$
Note that for $j\!<\!k$, $4\!\cdot \!\frac{1}{10^{k-j}}\!\le\! 4\!\cdot\!\frac{1}{10}\!=\!\frac{2}{5}$. 
As $\rho(w_j,w_k)\!=\!\rho(w_k,w_j)$, the bound $4\!\cdot\! \frac{1}{10^{|j-k|}}$ holds for all $j\!\ne\! k$. So for each fixed $j\in \mN$, 
\begin{equation}
\textstyle
\sum\limits_{k\in \mN\setminus \{ j\} }(1-\rho(w_j,w_k)^2) 
\le 2\sum\limits_{m=1}^\infty 4\!\cdot \!\frac{1}{10^{m}}
= 8\!\cdot\!\frac{\frac{1}{10}}{1-\frac{1}{10}} 
= \frac{8}{9} =: \eta\in (0,1). \tag{$2$}
\end{equation}

\medskip

\noindent {\bf Step 5.} (Bessel's inequality.)

\noindent Since all $w_j$ are real, the proof of Lemma~\ref{lem:D} shows that 
$$
\textstyle
\langle e_{w_j},e_{w_k}\rangle
=\pi(1-\rho(w_j,w_k)^2).
$$
For any finite $J\subset\mN$ and scalars $(c_j)_{j\in J}$,  we get 
$$
\textstyle 
\Big\|\sum\limits_{j\in J} c_j e_{w_j}\Big\|_{A^2}^2 
=
  \pi\sum\limits_{j\in J}|c_j|^2 + \pi\!\sum\limits_{\substack{(j,k)\in J\times J\\j\ne k}} \!
  c_j \bar c_k
(1-\rho(w_j,w_k)^2).
$$
Using $|c_j||c_k|\!\le\!\frac{1}{2}(|c_j|^2+|c_k|^2)$, the symmetry $\rho(w_j,w_k)\!=\!\rho(w_k,w_j)$, and  the estimate from ($2$) above, we can bound the second summand on the right-hand side as follows:
$$
\textstyle 
\!\!\!\!\begin{array}{rcl}
&& \Big|\sum\limits_{\substack{(j,k)\in J\times J\\j\ne k}} c_j \bar c_k
(1-\rho(w_j,w_k)^2)\Big|
\\[0.81cm]
\!\!\!&\le&\!\!\! 
\frac{1}{2} \sum\limits_{\substack{(j,k)\in J\times J\\j\ne k}} (|c_j|^2+|c_k|^2) 
(1-\rho(w_j,w_k)^2)
\\[0.81cm]
\!\!\!&\le&\!\!\! 
\frac{1}{2}\sum\limits_{j \in J}|c_j|^2\! \sum\limits_{k\in J\setminus \{j\}} (1-\rho(w_j,w_k)^2)
+
\frac{1}{2}\sum\limits_{k \in J}|c_k|^2 \!\sum\limits_{j\in J\setminus \{k\}} (1-\rho(w_j,w_k)^2)
\\[0.81cm]
\!\!\!&\le&\!\!\! \eta \sum\limits_{j\in J}|c_j|^2.
\end{array}
$$
Thus 
$$
\textstyle 
\Big\|\sum\limits_{j\in J} c_j e_{w_j}\Big\|_{A^2}^2 
\le 
  \pi\sum\limits_{j\in J}|c_j|^2 + \pi\eta \sum\limits_{j\in J}|c_j|^2
  = \pi(1+\eta)\sum\limits_{j\in J}|c_j|^2.
$$
For $h\in A^2$, set $c_j:=\overline{\langle h,e_{w_j}\rangle}$. By Cauchy-Schwarz and the above,
$$
\textstyle 
S_J:=
\sum\limits_{j\in J}|\langle h,e_{w_j}\rangle|^2 
= \Big\langle h,\sum\limits_{j\in J}\bar c_j e_{w_j}\Big\rangle \le \|h\|_{A^2}\sqrt{\pi(1+\eta)}\sqrt{S_J}.
$$
Hence $S_J \le \pi(1+\eta)\|h\|_{A^2}^2$. As $J\subset \mN$ was arbitrary, we obtain 
\begin{equation}
\textstyle 
\sum\limits_{k=1}^\infty |\langle h,e_{w_k}\rangle|^2 \le \pi(1+\eta)\|h\|_{A^2}^2\quad \text{for all } h\in A^2. \tag{$3$}
\end{equation}

\medskip

\noindent {\bf Step 6.} ($g\notin A^2$.)

\noindent Suppose $g\in A^2$. By Lemma~\ref{lem:D}, $\langle g,e_{w_j}\rangle \!=\! \pi(1-a_j^2)g(w_j)$. By the locally uniform convergence established in Step~3, 
$$
\textstyle 
g(w_j)=\sum\limits_{k=1}^\infty e_{w_k}(w_j).
$$
 Now $(1-a_j^2)e_{w_j}(w_j)\!=\!1$, and for $k\!\ne\! j$, using ($\star$), 
 we get 
 $$
 \textstyle (1-a_j^2)e_{w_k}(w_j)
 =\frac{(1-a_j^2)(1-a_k^2)}{(1-a_ja_k)^2}
 \overset{(\star)}{=}1-\rho(w_j,w_k)^2>0.
 $$
  So using (2), we obtain
$$
\textstyle 
|(1-a_j^2)g(w_j) - 1| = \Big|\sum\limits_{k\in \mN\setminus\{ j\}}(1-a_j^2)e_{w_k}(w_j)\Big| 
\le \sum\limits_{k\in \mN\setminus\{ j\}}(1-\rho(w_j,w_k)^2) \le \eta.
$$
Hence $(1-a_j^2)|g(w_j)| \ge 1-\eta$ for each $j\in \mN$. Thus
$$
\textstyle 
|\langle g,e_{w_j}\rangle| = \pi(1-a_j^2)|g(w_j)| 
\ge \pi(1-\eta) = \pi\!\cdot\!\frac{1}{9} > 0  \text{ for all } j\in \mN.
$$
Consequently, $\sum\limits_{j=1}^\infty |\langle g,e_{w_j}\rangle|^2\!=\!\infty$, contradicting ($3$). Hence $g\notin A^2$.

\medskip

\noindent {\bf Step 7.} ($r_1 g\!\in \! A^2$ and $r_2 g\!\in\! A^2$).

\noindent 
Using Lemma~\ref{lem:B} (with $\varphi\!=\!r_1$ and $w\!=\!w_k$), 
 and noting that every factor
$(\epsilon_{u_j}b_{u_j})^{m_j}$ in the Blaschke product $r_1$ has modulus at most $1$ on
$\mD$, we may discard all factors except the $k^{\text{th}}$ one to obtain an upper bound. Hence
$$
\textstyle 
\begin{array}{rcl}
\int_\mD |r_1(z)|^2|e_{w_k}(z)|^2\,dA(z) 
\!\!\!&=&\!\!\! \int_\mD |r_1(b_{w_k}(\eta))|^2\,dA(\eta) 
\\[0.21cm]
\!\!\!&\le&\!\!\! \int_\mD |b_{u_k}(b_{w_k}(\eta))|^{2m_k}\,dA(\eta).
\end{array}
$$
By Lemma~\ref{lem:C},  we get $|b_{u_k}(b_{w_k}(\eta))| \!=\! |b_{\zeta}(\eta)|$ for some $\zeta\in \mD$ satisfying $|\zeta|\!=\!\rho(u_k,w_k)\!=\!\frac{1}{2}$, where the last equality follows from Step 1. Using Lemma~\ref{lem:A}, we thus obtain 
$$
\textstyle 
\int_\mD |r_1(z)|^2|e_{w_k}(z)|^2\,dA(z) 
\le \frac{\pi}{m_k+1}(\frac{1+\frac{1}{2}}{1-\frac{1}{2}})^2 = \frac{9\pi}{m_k+1}.
$$
As $r_1 e_{w_k}\in \calO(\mD)$ with finite $L^2(\mD,dA)$-norm, we have  $r_1e_{w_k}\in A^2$. Also, since $m_k\!=\!k^3$, we get 
$$
\textstyle \|r_1e_{w_k}\|_{A^2}\le\sqrt{\frac{9\pi}{m_k+1}}=  \frac{3\sqrt\pi}{\sqrt{k^3+1}} \le \frac{3\sqrt\pi}{k^{\frac{3}{2}}}.
$$ 
Thus
$$
\textstyle 
\sum\limits_{k=1}^\infty \|r_1e_{w_k}\|_{A^2} 
\le 3\sqrt\pi\sum\limits_{k=1}^\infty\frac1{k^{\frac{3}{2}}} 
<\infty.
 $$
 Therefore the series 
 $$
 \textstyle \sum\limits_{k=1}^\infty r_1e_{w_k}
 $$ 
 converges absolutely in the Banach space $A^2$, and hence converges in $A^2$ to some element $h\in A^2$.

 Since point evaluations  are continuous linear functionals on $A^2$, the above  
$A^2$-norm convergence implies pointwise convergence on $\mD$. So for each $z\in\mD$, we conclude that 
$$
\textstyle 
h(z) = \sum\limits_{k=1}^\infty r_1(z)e_{w_k}(z) 
= r_1(z)\sum\limits_{k=1}^\infty e_{w_k}(z) 
= r_1(z)g(z),
$$
using the definition of $g$ from Step 3. Hence $r_1 g \!= \! h \in A^2$. Similarly, $r_2 g\in A^2$.
\end{proof}

\subsection{Proof of the non-flatness of the $H^\infty$-module $A^2$}

\begin{theorem}
\label{thm_1}
There exist $r_1,r_2\in H^\infty$ and $f_1,f_2\in A^2$ with 
$$
\textstyle 
r_1f_1+r_2f_2=0,
$$ 
such that there do not exist $k\in \mN,$ $\theta_{ij}\in H^\infty$ $(1\!\le\! i\!\le\! 2,\,1\!\le\! j\!\le\! k),$ 
and $\mu_1,\cdots, \mu_k \in A^2$ satisfying 
\begin{equation}
\textstyle 
\left. \begin{array}{ll}
\sum\limits_{i=1}^2 r_i\theta_{ij}=0\text{ \em for all }j\in \{1,\cdots, k\},
\text{ \em and }\\
f_i = \sum\limits_{j=1}^k \theta_{ij}\mu_j \text{ \em for all }i\in \{1,2\}.
\end{array}\right\} \tag{$\ast$}
\end{equation}
\end{theorem}
\begin{proof} From Proposition~\ref{lem:main}, there exist Blaschke products
 $r_1,r_2$  with disjoint zero sets, and a $g\in \calO(\mD)$  such that  $r_1g,r_2g\in A^2$ but $g\notin A^2$. Set $f_1\!:=\!r_2 g\in A^2$ and $f_2\!:=\!-r_1g\in A^2$.
Then 
$$
\textstyle 
r_1f_1+r_2f_2 = r_1(r_2g)+r_2(-r_1g)=0. 
$$
Suppose, on the contrary,  that there exist $k\in \mN,$ functions $\theta_{ij}\in H^\infty$ ($1\!\le\! i\!\le\!2$, $1\!\le\! j\!\le\! k$), and $\mu_1,\cdots, \mu_k \in A^2$ satisfying ($\ast$). 
By Corollary~\ref{cor:collapse}, there exists a $\widetilde{g}\in A^2$ with $
f_1 \!=\! -r_2\widetilde{g}$ and $ f_2\!=\!r_1\widetilde{g}$.
Comparing $f_1\!=\!r_2g$ with $f_1\!=\!-r_2\widetilde{g}$, we obtain
$r_2(g+\widetilde{g})\!=\!0$ 
on $\mD$. Since $r_2$ is not identically zero, the identity theorem gives
$g\!=\!-\widetilde{g}$ on $\mD$. Thus $g\!=\!-\widetilde{g}\in A^2$, contradicting Proposition~\ref{lem:main}. Hence no such representation exists.
\end{proof}

\begin{remark}[$A^2$ versus $H^2$]
We contrast the above with the fact that 
the $H^\infty$-module $H^2$ is flat (see~\cite[Proposition~3.1]{Sas}).
The algebraic reduction of \S\ref{subsection_algebraic_reduction}  applies also to $H^2$ (instead of $A^2$).
Thus the non-flatness of $A^2$ arises from the results of  \S\ref{subsection_construction}.
In particular, thanks to the fact that the $A^2$-norm is
an area integral, Lemma~\ref{lem:A} produces a 
decay of order $1/\sqrt{m}$ when a Blaschke factor is raised to the power $m$. Such a decay
is absent in the $H^2$ case, since for an inner $r$,  $|r(\zeta)|\!=\!1$ for almost every
$\zeta \in \mathbb{T}$, showing that multiplication by $r$ is an
isometry of $H^2$ (i.e., $  \|rf\|_{H^2} \!=\! \|f\|_{H^2}
  $ for all $ f \in H^2$). Consequently, the
mechanism underlying Proposition~\ref{lem:main} 
(accumulating reproducing kernels near the boundary while suppressing
their norm using high-order Blaschke factors) is 
unavailable for $H^2$. The aforementioned `isometric rigidity' is captured
by the vector-valued Beurling theorem for $H^2$ (every closed $H^\infty$-invariant subspace $S$ of $(H^2)^n$ is 
 the isometric image of $(H^2)^k$ under multiplication by a matricial inner function $\Theta\in (H^\infty)^{n\times k}$ 
 for some $k \!\le\! n$), and underlies the proof of the flatness of $H^2$.
On the other hand,  an inner function acts only as
a contraction on the Bergman space, and so no such isometric classification
of invariant subspaces is available. The resulting extra `room'
is exploited by Proposition~\ref{lem:main}.
\end{remark}

\section{Maximal ideals $\mathfrak{m}$ of $H^\infty$ such that $\mathfrak{m}A^2=A^2$}
\label{section_not_faithful}

\subsection{Maximal ideal space of $H^\infty$}$\;$

\noindent 
Background on the Hardy algebra $H^\infty$ and its maximal ideal space  can be found, for example, in \cite{Gar} and \cite{Hof}. 

 Let $A$ be a commutative unital complex semisimple Banach algebra. 
 The dual space $A^*$ of $A$ consists of all continuous linear complex-valued maps defined on $A$. The {\em maximal ideal
space} $M(A)$ of $A$  is the set of all nonzero homomorphisms\footnote{That is, nonzero  multiplicative continuous linear functionals.} $\varphi\!:\! A \!\to\!  {\mathbb{C}}$. As $M(A)$ is a subset of $A^*$, it inherits the  weak-$\ast$ topology of $A^*$, called the {\em Gelfand topology} on $M(A)$. 
 There is a bijective correspondence between $M(A)$ and the set of all maximal ideals of $A$: Each  $\varphi \in M(A)$ corresponds to the maximal ideal $\mathfrak{m}\!:=\!\ker \varphi $ of $A$.  Equipped with the Gelfand topology, $M(A)$ is a compact Hausdorff space. The set $M(A)$ is contained in the unit sphere of  the Banach space ${\mathcal{L}}(A,{\mathbb{C}})$ of all complex-valued continuous linear maps on $A$  with the operator norm (given by $\|\varphi\|\!=\!\sup_{a\in A, \;\|a\|\le 1} |\varphi(a)|$ for all $\varphi \in {\mathcal{L}}(A,{\mathbb{C}})$).  
 Let $C(M(A))$ denote the Banach algebra  of 
complex-valued continuous functions on $M(A)$ with pointwise operations and the supremum norm. 
 The {\em Gelfand transform} $\widehat{a}\in C(M(A))$ of an element $a\in A$ is  
 defined by 
 $$
 \widehat{a}(\varphi)\!:=\!\varphi(a)\text{ for all }\varphi \in M(A).  
$$
 Consider the identity function ${\bm{z}} \in H^\infty$, given by 
 $$
 \textstyle {\mathbb{D}}\owns \zeta \mapsto {\bm{z}}(\zeta):=\zeta.
 $$
  Then the map $\pi: M( H^\infty)\to {\mathbb{C}}$, 
 $$
 \textstyle 
\pi(\varphi)=\varphi({\bm{z}})\text{ for all }\varphi \in  M( H^\infty),
$$ 
 is a continuous map onto the closed  unit disk $\overline{{\mathbb{D}}}$  in ${\mathbb{C}}$. 
 For $\lambda\in \mD$, let $\varphi_\lambda\in M(H^\infty)$ be the point evaluation map at $\lambda$, given by 
 $$
 \textstyle \varphi_\lambda(g)=g(\lambda)\text{ for all }g\in H^\infty.
 $$ 
 Then $\pi(\varphi_\lambda)\!=\!\lambda$ for all $\lambda \in \mD$.  
 Over  $\pi^{-1}{\mathbb{D}}$, $\pi$ is one-to-one, and maps $\pi^{-1}{\mathbb{D}}$ homeomorphically onto $\mD$. 
   The remaining part, $M( H^\infty)\setminus (\pi^{-1}\mD)$, is mapped by $\pi$ onto the unit circle. If $|\alpha|\!=\!1$, then $\pi^{-1}\{\alpha\}$ is the {\em fibre of $M( H^\infty)$ over $\alpha$}. We will use the following result due to I.J. Schark\footnote{An acronym for real authors’  names:
{\bf{I}}rving Kaplansky, 
{\bf{J}}ohn Wermer, 

{\bf{S}}hiuzo Kakutani, 
{\bf{C}}reighton Buck, 
{\bf{H}}alsey Royden, 
{\bf{A}}ndrew Gleason, 

{\bf{R}}ichard~Arens, and 
{\bf{K}}enneth Hoffman.} (see \cite{Sch}).

   \begin{proposition}
   \label{prop_Schark}
   If $g\in  H^\infty$ and $\alpha \in {\mathbb{T}},$ then 
   $$
   \textstyle 
   \!\!\begin{array}{rcl}
   \text{\em ran}(\widehat{g}|_{\pi^{-1}\{\alpha\}})
   \!\!\!\!\!&:=&\!\!\!\!
   \{\varphi(g)\;|\; \varphi \in \pi^{-1}\{\alpha\}\}
   \\[0.21cm]
   \!\!\!\!\!&=&\!\!\!\!\!
   \Big\{\zeta \in {\mathbb{C}}\;\Big|\!
   {\scaleobj{0.93}{\begin{array}{ll} \text{\em there exists a sequence }(\lambda_n)_{n\in {\mathbb{N}}}\text{ \em in }{\mathbb{D}}
   \text{ \em satisfying } \\
\lim\limits_{n\to \infty}\lambda_n\!=\!\alpha \text{ \em and } 
\lim\limits_{n\to \infty} g(\lambda_n)\!=\!\zeta
\end{array}}}\!\!\!\Big\}.
\end{array}
$$
\end{proposition}

\noindent In particular, taking $\alpha\!=\!\{1\}$, if $g\in H^\infty$ satisfies 
 $$
 \textstyle 
\lim\limits_{\mD\owns z\to 1} g(z)\!=\!0,
$$ 
then 
$$
\textstyle  \Big\{\zeta \in {\mathbb{C}}\;\Big|\!
   {\scaleobj{0.93}{\begin{array}{ll} \text{ there exists a sequence }(\lambda_n)_{n\in {\mathbb{N}}}\text{ in }{\mathbb{D}}
   \text{ satisfying } \\
\lim\limits_{n\to \infty}\lambda_n\!=\!1 \text{ and } 
\lim\limits_{n\to \infty} g(\lambda_n)\!=\!\zeta.
\end{array}}}\!\!\!\Big\}=\{0\},
$$
and so Proposition~\ref{prop_Schark} implies $g\in \ker \varphi$ for {\em all} $\varphi \in \pi^{-1}\{1\}$. 

\goodbreak 

\subsection{The result and its proof}

\begin{theorem}
\label{thm_2}
Let $\alpha\in \mathbb{T},$ $\varphi\in \pi^{-1}\{\alpha\},$ and $\mathfrak{m}\!=\!\ker \varphi$. 
Then $\mathfrak{m}A^2\!=\!A^2$. 
\end{theorem}
\begin{proof} There is rotational symmetry, and so we will just consider $\alpha\!=\!1$. 
We need to show that $A^2\subset \mathfrak{m}A^2$. 
Let $f \in A^2$ be such that $\|f\|_{A^2}\!=\!1$. We will construct an outer $g\in H^\infty$ such that $g\in \mathfrak{m}$ and $\frac{f}{g}\in A^2$. 

The Poisson kernel is given by 
$$
\textstyle P_z(t) = \frac{1-|z|^2}{|e^{it}-z|^2}\;\;\text{ for }z\in \mD\;\text{ and }\;t\in (-\pi,\pi].
$$
Define $\phi: (-\pi,\pi]\to [0,+\infty]$ by 
$$
\textstyle
\phi(t) = \int_{\mD} |f(z)|^2 \frac{P_z(t)}{2\pi}\, dA(z).
$$
As $\int_{-\pi}^\pi \frac{P_z(t)}{2\pi} dt\!=\!1$ and $\|f\|_{A^2}\!=\!1$, we get  
$$
\textstyle 
\int_{-\pi}^\pi \phi(t) \,dt =\int_\mD |f(z)|^2 \,dA(z)\cdot 1=1.
$$
Thus $\phi\in L^1(\pi,\pi)$, and so $\phi(t)\! <\!\infty$ for almost all $t\in (-\pi,\pi)$. 
 
Since $\|f\|_{A^2}\!=\!1$, the holomorphic function $f$ cannot be identically zero. As the continuous map $(z,t)\mapsto P_z(t)$ is pointwise positive, we have $0\!<\!\phi(t) \!<\!\infty$ for almost all $t\in (-\pi,\pi]$. 

 Define $\Phi:(-\pi,\pi]\to \mathbb{R}$ by 
$$
\textstyle
 \Phi(t) = \int_0^t \phi(\tau)\, d\tau.
$$
Then $\Phi$ is absolutely continuous, and hence also continuous. Moreover,  $\Phi$ is differentiable almost everywhere with $\Phi'(t)\!=\! \phi(t)\in (0,\infty)$ for almost all $t\in (-\pi,\pi]$. 
As $\phi\!>\!0$ almost everywhere, we have $\Phi(t)\!\neq\! 0$ for $t\in (-\pi,\pi]\setminus \{0\}$. Clearly $\Phi(0)\!=\!0$. 
Since $\phi$ is nonnegative everywhere, it follows from the above that $|\Phi(t)|\!\le\! 1\!=\!\int_{-\pi}^\pi \phi(t) \,dt$. 
Set
$$
\textstyle
\begin{array}{rcl}
\omega(t) \!\!\!\!&=&\!\!\!\!  \max\{|\Phi(t)|, \frac{|t|}{\pi} \}, \\[0.21cm] 
W(e^{it}) \!\!\!\!&=&\!\!\!\! \sqrt[4]{\omega(t)}.
 \end{array}
$$
Then $\omega(t)\in [0,1]$ for all $t\in(-\pi,\pi]$. 
We claim that $\log W \in L^1(\mT)$. 
We have
 $
\textstyle 
|\log W(e^{it})| = |\tfrac{1}{4}\log \omega(t)|
\le \tfrac{1}{4}|\log \frac{|t|}{\pi}|.
$ 
Hence
$$
\textstyle
\int_{-\pi}^{\pi} |\log W(e^{it})|\, dt
\le \tfrac{1}{4}2\int_0^{\pi} \log\frac{\pi}{t}\, dt 
= \frac{1}{2}(t\log \pi +t-t\log t)|_0^{\pi}  =\frac{\pi}{2} <\infty.
$$
So $\log W \in L^1(\mT)$.

Define the outer function $g$ by 
$$
\textstyle 
g(z) = \exp(\frac{1}{2\pi}\int_{-\pi}^{\pi} \frac{e^{it}+z}{e^{it}-z} \log W(e^{it})\, dt)\text{ for all }z\in \mD.
$$
As $\|W\|_\infty\!\le\! 1$, it follows that $g\in H^\infty$.

 Moreover, we have $\omega(t)\to 0$ as $|t|\to 0$, and so $g(e^{it})\to 0$ as $|t|\to 0$. 
By Lindel\"of's theorem (see, e.g., \cite[Ex.~7(c),  p.88-89]{Gar}), we obtain 
$$
\textstyle
\lim\limits_{\mathbb{D} \ni z \to 1} g(z) = 0.
$$
Thus by Proposition~\ref{prop_Schark}, we conclude that $g\in \mathfrak{m}$. 
It remains to show that $\frac{f}{g}\in A^2$, which we establish below. 

The modulus of $g$ satisfies
$$
\textstyle
\frac{1}{|g(z)|^2}
\!=\! \exp(-2 \int_{-\pi}^{\pi} \frac{P_z(t)}{2\pi}\log W(e^{it})\, dt)
\!=\! \exp(\int_{-\pi}^{\pi} \frac{P_z(t)}{2\pi} \!\cdot\!\log\frac{1}{(W(e^{it}))^2}\, dt).
$$
Applying Jensen's inequality to the convex function $\mR\owns x\mapsto e^{x}$ with the probability measure $\frac{P_z(t)}{2\pi}\, dt$ gives
$$
\textstyle
\frac{1}{|g(z)|^2}
\!\le\! \int_{-\pi}^{\pi} \frac{P_z(t)}{2\pi}\, e^{\log\frac{1}{(W(e^{it}))^2}}\, dt
\!=\! \int_{-\pi}^{\pi} \frac{P_z(t)}{2\pi} \frac{1}{W(e^{it})^2}\, dt
\!=\! \int_{-\pi}^{\pi} \frac{P_z(t)}{2\pi} \frac{1}{\sqrt{\omega(t)}}\, dt.
$$
 Thus we get 
 $$
 \textstyle 
 \int_{\mD} \frac{|f(z)|^2}{|g(z)|^2}\, dA(z)
\le \int_{\mD} |f(z)|^2 (\int_{-\pi}^{\pi}
  \frac{P_z(t)}{2\pi} \frac{1}{\sqrt{\omega(t)}}\, dt)\, dA(z).
$$
Since all terms are non-negative, Fubini's theorem yields
$$
 \textstyle 
\int_{\mD} \frac{|f(z)|^2}{|g(z)|^2}\, dA(z)
\le  \int_{-\pi}^{\pi} \frac{1}{\sqrt{\omega(t)}}
  (\int_{\mD} |f(z)|^2 \frac{P_z(t)}{2\pi} \, dA(z))\, dt 
  = \int_{-\pi}^{\pi} \frac{\phi(t)}{\sqrt{\omega(t)}}\, dt.
$$
 For $t\!>\!0$,  we have $\omega(t) \ge \Phi(t) \!=\! \int_0^t \phi(\tau)\, d\tau\!\ge\! 0$, giving
$$
 \textstyle 
\frac{\phi(t)}{\sqrt{\omega(t)}} \le \frac{\phi(t)}{\sqrt{\Phi(t)}} \stackrel{\text{(a.e.)}}{=} \frac{\Phi'(t)}{\sqrt{\Phi(t)}},
$$
 and so 
$$
\textstyle 
 0\le \int_{0}^{\pi} \frac{\phi(t)}{\sqrt{\omega(t)}}\, dt 
\le \int_0^{\pi} \frac{\Phi'(t)}{\sqrt{\Phi(t)}}\, dt
= 2\sqrt{\Phi(t)}|_0^{\pi}
=2\Phi(\pi)\le 2(1)
< \infty.
$$
 For $t\!<\!0$,  we have $\omega(t) \ge|\Phi(t)|\!=\! -\Phi(t) \!=\! -\int_0^t \phi(\tau)\, d\tau\!\ge\! 0$, giving
$$
 \textstyle 
\frac{\phi(t)}{\sqrt{\omega(t)}} \le \frac{\phi(t)}{\sqrt{-\Phi(t)}} \stackrel{\text{(a.e.)}}{=}  \frac{\Phi'(t)}{\sqrt{-\Phi(t)}},
$$
and so 
$$
\textstyle 
\!\!\!\!\!\!\begin{array}{rcl}
 0\le \int_{-\pi}^{0} \frac{\phi(t)}{\sqrt{\omega(t)}}\, dt 
\!\!\!&\le&\!\!\! \int_{-\pi}^0 \frac{\Phi'(t)}{\sqrt{-\Phi(t)}}\, dt
= -2\sqrt{-\Phi(t)}|_{-\pi}^{0}
=2\sqrt{-\Phi(-\pi)}\\
\!\!\!&\le&\!\!\! 2(1)
< \infty.
\end{array}
$$
Consequently,  $\frac{f}{g} \in A^2$.
\end{proof}

\smallskip 

\noindent {\bf Acknowledgement.} The author thanks Professor Alexandru Aleman (Lund University) for suggesting the investigation of this question, and for several useful discussions and  insights.


\medskip 
 
 
\begin{thebibliography}{99}

\bibitem{BakNew}
J. Bak and D. Newman. 
{\em Complex analysis}. Third edition. 
 Undergraduate Texts in Mathematics, Springer, 2010.
      
\bibitem{Bos}
S. Bosch. 
{\em Algebraic geometry and commutative algebra}. Second edition. 
   Universitext, Springer, 2022. 
      
\bibitem{Bou}
N. Bourbaki. 
{\em Commutative algebra. Chapters 1-7}.  
 Elements of Mathematics, Springer, 1998.      
      
\bibitem{Gar}
J. Garnett. 
{\em Bounded analytic functions}. Revised first edition. 
Graduate Texts in Mathematics 236, Springer, 2007.

\bibitem{HedKorZhu}
H. Hedenmalm, B. Korenblum, and K. Zhu. 
{\em Theory of Bergman spaces}.  
Graduate Texts in Mathematics 199, Springer, 2000. 

\bibitem{Hof}
K. Hoffman. 
{\em Banach spaces of analytic functions}. 
Dover, 1962.

\bibitem{Nes}
J. Nestruev. 
{\em Smooth manifolds and observables}. Second edition. 
Graduate Texts in Mathematics 220, 
    Springer, 2020.

\bibitem{Nik}
N. Nikolski. 
{\em Treatise on the shift operator}.
    Grundlehren der mathematischen Wissenschaften 273.
     Springer, 1986.
     
\bibitem{Sas}
A. Sasane. 
On the faithful flatness of some modules arising in analysis. 
{\em Nagoya Mathematical Journal}, 261, Paper No. e19, 13, 2026. 

\bibitem{Sch}
I.J. Schark. 
Maximal ideals in an algebra of bounded analytic functions. 
{\em Journal of Mathematics and Mechanics}, 10:735-746, no. 5, 1961.

\bibitem{Ser}
J-P. Serre. 
G\'{e}om\'{e}trie alg\'{e}brique et g\'{e}om\'{e}trie analytique. 
{\em Annales de l'Institut Fourier, Universit\'{e} de Grenoble}, 
6:1-42, 1955/56.

\bibitem{Vak}
R. Vakil. 
{\em The rising sea -- foundations of algebraic geometry}. 
Princeton University Press, 2025.

  
\end{thebibliography}
\end{document}